\newif\ifMariaHasTheToken
\MariaHasTheTokenfalse                               

\documentclass[12pt]{amsart}
\usepackage[utf8]{inputenc}
\usepackage{thm-restate}

\usepackage[T1]{fontenc}
\usepackage[utf8]{inputenc}

\usepackage[dvipsnames]{xcolor}
\usepackage[footskip=1cm, headheight = 16pt, top=3.7cm, bottom=3cm,  right=2.5cm,  left=2.5cm ]{geometry}
\usepackage[colorlinks,linkcolor=blue,citecolor=red]{hyperref}
\usepackage{geometry}
\usepackage{amsmath}
\usepackage{comment}
\usepackage{subcaption}
\usepackage{todonotes}
\usepackage{amssymb}
\usepackage{amsthm}
\usepackage{enumerate}
\usepackage[capitalise]{cleveref}
\usepackage{hyperref}
\usepackage{graphicx}
\usepackage[inline]{enumitem}
\usepackage{setspace}

\usepackage{tikzit}

\tikzstyle{box}=[shape=rectangle, text height=1.5ex, text depth=0.25ex, yshift=0.5mm, fill=white, draw=black, minimum height=5mm, yshift=-0.5mm, minimum width=5mm, font={\small}]
\tikzstyle{gate}=[shape=rectangle, text height=1.5ex, text depth=0.25ex, yshift=0.5mm, fill=white, draw=black, minimum height=5mm, yshift=-0.5mm, minimum width=5mm, font={\small}, tikzit category=circuit]
\tikzstyle{big gate}=[shape=rectangle, text height=1.5ex, text depth=0.25ex, yshift=0.5mm, fill=white, draw=black, minimum height=10mm, yshift=-0.5mm, minimum width=5mm, font={\small}, tikzit category=circuit]
\tikzstyle{Z dot}=[inner sep=0mm, minimum size=2mm, shape=circle, draw=black, fill={rgb,255: red,221; green,255; blue,221}, tikzit category=zx]
\tikzstyle{Z phase dot}=[minimum size=5mm, font={\footnotesize\boldmath}, shape=rectangle, rounded corners=2mm, inner sep=0.2mm, outer sep=-2mm, scale=0.8, tikzit shape=circle, draw=black, fill=white, tikzit draw=blue, tikzit category=zx]
\tikzstyle{X dot}=[Z dot, shape=circle, draw=black, fill={rgb,255: red,255; green,136; blue,136}, tikzit category=zx]
\tikzstyle{X phase dot}=[Z phase dot, tikzit shape=circle, tikzit draw=blue, fill={rgb,255: red,255; green,136; blue,136}, font={\footnotesize\boldmath}, tikzit category=zx]
\tikzstyle{hadamard}=[fill=yellow, draw=black, shape=rectangle, inner sep=0.6mm, minimum height=1.5mm, minimum width=1.5mm, tikzit category=zx]
\tikzstyle{paulibox}=[fill={rgb,255: red,221; green,221; blue,255}, draw=black, shape=rectangle, inner sep=0.6mm, minimum height=5mm, minimum width=5mm, font={\footnotesize}, text height=1.5ex, text depth=0.25ex, tikzit category=zx]
\tikzstyle{vertex}=[inner sep=0mm, minimum size=1mm, shape=circle, draw=black, fill=black, tikzit category=misc]
\tikzstyle{vertex set}=[inner sep=0mm, minimum size=2mm, shape=circle, draw=black, fill=white, font={\footnotesize\boldmath}, tikzit category=misc]
\tikzstyle{small black dot}=[fill=black, draw=black, shape=circle, inner sep=0pt, minimum width=1.2mm, tikzit category=circuit]
\tikzstyle{cnot ctrl}=[fill=black, draw=black, shape=circle, inner sep=0pt, minimum width=1.2mm, tikzit category=circuit]
\tikzstyle{cnot targ}=[fill=white, draw=white, shape=circle, tikzit category=circuit, label={center:$\oplus$}, inner sep=0pt, minimum width=2.1mm, tikzit fill={rgb,255: red,102; green,204; blue,255}, tikzit draw=black]
\tikzstyle{ket}=[fill=white, draw=black, shape=regular polygon, regular polygon sides=3, regular polygon rotate=-30, scale=0.7, inner sep=1pt, tikzit category=circuit, tikzit shape=rectangle, tikzit fill=green]
\tikzstyle{bra}=[fill=white, draw=black, shape=regular polygon, regular polygon sides=3, regular polygon rotate=30, scale=0.7, inner sep=1pt, tikzit category=circuit, tikzit shape=rectangle, tikzit fill=red]
\tikzstyle{scalar}=[shape=rectangle, text height=1.5ex, text depth=0.25ex, yshift=0.5mm, fill=white, draw=black, minimum height=5mm, yshift=-0.5mm, minimum width=5mm, font={\small}]
\tikzstyle{clabel}=[fill=white, draw=none, shape=rectangle, tikzit fill={rgb,255: red,56; green,255; blue,242}, font={\footnotesize}, inner sep=1pt, tikzit category=labels]
\tikzstyle{empty diagram}=[draw={gray!40!white}, dashed, shape=rectangle, minimum width=1cm, minimum height=1cm, tikzit category=misc]
\tikzstyle{white dot}=[Z dot, fill=none]
\tikzstyle{gray dot}=[X dot]
\tikzstyle{white phase dot}=[Z phase dot]
\tikzstyle{gray phase dot}=[X phase dot]
\tikzstyle{small hadamard}=[hadamard]
\tikzstyle{implies}=[-implies, double, double distance=2pt]
\tikzstyle{tiny dot}=[fill=black, draw=black, shape=circle, scale=0.15]

\tikzstyle{simple}=[-]
\tikzstyle{hadamard edge}=[-, dashed, dash pattern=on 2pt off 0.5pt, thick, draw={rgb,255: red,68; green,136; blue,255}]
\tikzstyle{box edge}=[-, dashed, dash pattern=on 2pt off 0.5pt, thick, draw={rgb,255: red,203; green,192; blue,225}]
\tikzstyle{brace edge}=[-, tikzit draw=blue, decorate, decoration={brace,amplitude=1mm,raise=-1mm}]
\tikzstyle{diredge}=[->]
\tikzstyle{double edge}=[-, double, shorten <=-1mm, shorten >=-1mm, double distance=2pt]
\tikzstyle{gray edge}=[-, draw=gray]
\tikzstyle{pointer edge}=[->, very thick, gray]
\tikzstyle{boldedge}=[-, line width=1.1pt, shorten <=-0.17mm, shorten >=-0.17mm]
\tikzstyle{forest edge}=[-, draw={rgb,255: red,119; green,187; blue,55}]
\tikzstyle{grey fill}=[-, fill={rgb,255: red,174; green,174; blue,174}, opacity=0.5]
\tikzstyle{sky edge}=[-, draw={rgb,255: red,76; green,156; blue,208}]
\tikzstyle{violet edge}=[-, draw={rgb,255: red,156; green,106; blue,176}]
\tikzstyle{light grey}=[-, {gray!50!white}, tikzit fill={rgb,255: red,164; green,164; blue,164}]

\input{graph.tikzdefs}

\newtheorem{theorem}{Theorem}[section]

\newtheorem{lemma}[theorem]{Lemma}

\newtheorem{observation}[theorem]{Observation}
\newtheorem{claim}[theorem]{Claim}

\newtheorem{conjecture}[theorem]{Conjecture}

\newcommand*{\claimproofs}{Proof of Claim.\quad}
\newenvironment{claimproof}[1][\claimproofs]{\begin{proof}[#1]}{\end{proof}}
\newcounter{tbox}
\newcommand{\sta}[1]{\vspace*{0.3cm}\refstepcounter{tbox}\noindent{ \parbox{\textwidth}{(\thetbox) \emph{#1}}}\vspace*{0.3cm}}

\theoremstyle{remark}

\newcommand{\defn}[1]{\textcolor[RGB]{128,0,0}{\emph{#1}}}

\newcommand{\set}[1]{\left\{#1\right\}} 
\newcommand{\sm}{\setminus}
\newcommand{\mt}{\emptyset}

\newcommand{\cT}{\mathcal{T}}
\newcommand{\cC}{\mathcal{C}}
\newcommand{\dd}{\hbox{-}}
\DeclareMathOperator{\dist}{dist}
\DeclareMathOperator{\tw}{tw}
\DeclareMathOperator{\sep}{sep}

\title{Coarse Balanced Separators in Biclique-Induced-Minor-Free Graphs}

\author{Maria Chudnovsky$^*$}
\thanks{$^*$Princeton University, Princeton, NJ, USA. Supported by NSF Grants DMS-2348219 and CCF-2505100,  AFOSR grant FA9550-25-1-0275, and a Guggenheim Fellowship.}

\author{Julien Codsi$^\dagger$}
\thanks{$^\dagger$Princeton University, Princeton, NJ, USA. Supported by NSF Grant DMS-2348219 and by the Fonds de recherche du Québec via the doctoral research scholarship 321124.}
\author{Claire Kaneshiro$^\ddagger$}
\thanks{$^\ddagger$Princeton University, Princeton, NJ, USA}

\newcommand{\class}{\mathcal{C}_t}

\begin{document}

\begin{abstract}
It is a classical theorem of Robertson and Seymour (1986) that the treewidth of a graph is linearly related to its separation number: the smallest integer $k$ such that, for every weight function on the vertices, the graph admits a balanced separator of size at most $k$. Motivated by recent progress on coarse treewidth, Abrishami, Czy\.zewska, Kluk, Pilipczuk, Pilipczuk, and Rza\.zewski (2025) conjectured the following coarse analogue: for every $r\in \mathbb{N}$ there exists an $r'\in \mathbb{N}$ such that every graph that admits balanced separators that can be covered by a bounded number of balls of bounded radius $r$  
admits a tree decomposition where every bag can be covered by a bounded number of balls of radius $r'$. 
We verify a stronger variant of this conjecture for all $r \in \mathbb{N}$ for the hereditary class of $K_{t,t}$-induced-minor-free graphs of bounded clique number.
A key step in the proof is the following result, which we expect to be of independent interest. In $K_{t,t}$-induced-minor-free graphs with clique number bounded by $s$, given a large subset of vertices $Y \subseteq V(G)$, there is a set $Z$ whose size is bounded by a function polynomial in $s$, such that no ball of radius $r$ in $G- Z$ covers a large proportion of $Y$.
\end{abstract}

\maketitle

\section{Introduction}

All graphs in this paper are finite and simple. We refer the reader to \cref{subsection notation} for notation and definitions.
\defn{Treewidth} is a graph parameter which, informally, captures “how close” a graph is to a tree, and graphs with small treewidth are necessarily sparse.
Formally, a \defn{tree decomposition} $\cT = (T, \beta)$ of a graph $G$ consists of a tree $T$ and a map $\beta:V(T) \to 2^{V(G)}$ with the following properties:
\begin{enumerate}[label=(\roman*)]
    \item For every $v \in V(G)$, there exists $t \in V(T)$ such that $v \in \beta(t)$.
    \item For every $uv \in E(G)$, there exists $t \in V(T)$ such that $u, v \in \beta(t)$.
    \item For every $v \in V(G)$, the subgraph of $T$ induced by $\{t \in V(T) \mid v \in \beta(t)\}$ is connected.
\end{enumerate}
For each $t\in V(T)$, we refer to $\beta(t)$ as a \textit{bag of} $(T, \beta)$.  The \emph{width} of a tree decomposition $(T, \beta)$, denoted by $width(T, \beta)$, is $\max_{t \in V(T)} |\beta(t)|-1$. The \emph{treewidth} of $G$, denoted by $tw(G)$, is the minimum width of a tree decomposition of $G$.  Graphs of bounded treewidth are well-understood both structurally \cite{RS86II} and algorithmically \cite{Bodlaender1988DynamicTreewidth}.

Let $w: V(G) \to \R_{\geq 0}$ be a weight function on the vertices of $G$. For a subset $C\subseteq V(G)$, we set $w(C) = \sum_{v\in C}w(v)$.
A set $X \subseteq V(G)$ is a \defn{$w$-balanced separator} if $w(C) \leq  \frac{1}{2} w(V(G))$ for every connected component $C$ of $G \setminus X$. The \defn{separation number} of $G$, denoted by $\sep(G)$, is the minimum $k\in \N$ such that, for every weight function $w: V(G)\to \R_{\geq 0}$ there is a $w$-balanced separator $X$ with size at most $k$. Given a subset $Y \subseteq V(G)$, we say that $X$ is a \defn{balanced separator for $Y$} if $X$ is a $w$-balanced separator for the weight function $w$ given by $w(v) = 1$ for $v\in Y$ and $w(v)=0$ otherwise. It is a classical result of Robertson and Seymour \cite{RS86II} that there is a linear correspondence between treewidth and separation number. 

\begin{theorem}[\cite{RS86II}] \label{thm: twboundbysep}
     For every graph $G$, $ \sep(G) -1 \leq  \tw(G)  \leq 4\sep(G).$
\end{theorem}

This paper will investigate a ``coarse analogue'' of this theorem. 
For a set $S \subseteq V(G)$, we say that $\hat{S}$ is an \defn{$r$-cover} of $S$ if $S \subseteq N^r_G[\hat{S}]$, and  $\hat{S}$ is a \defn{minimum $r$-cover} of $S$ if there is no $r$-cover $\hat{T}$ of $S$ with $|\hat{T}| < |\hat{S}|$. The set $S$ is \defn{$(k,r)$-coverable} if there exists an $r$-cover $\hat{S}$ of $S$ with $|\hat{S}|\leq k$, and $S$
is \defn{internally $(k,r)$-coverable} if  $G[S]$ is $(k,r)$-coverable.
We say that $G$ \defn{admits $(k,r)$-balanced separators} if, for every weight function $w: V(G) \to \R_{\geq 0}$, there is a $(k,r)$-coverable $w$-balanced separator. If every induced subgraph of $G$ admits $(k,r)$-coverable balanced separators, we say that $G$ is \defn{$(k,r)$-separable}.
A class of graphs is \defn{$(k,r)$-separable} if every graph in it is $(k,r)$-separable. 
We say that a tree decomposition is \defn{(internally) $(k,r)$-coverable} if every bag is (internally) $(k,r)$-coverable. 

Abrishami, Czy\.zewska, Kluk, Pilipczuk, Pilipczuk, and Rza\.zewski proposed the following coarse counterpart to Theorem~\ref{thm: twboundbysep}.

\begin{conjecture} [Conjecture 1.1 in \cite{ACKPPR25}] \label{conj: coarse bal sep and tw}
 For every $k,r \in \N$, there exist $k',r' \in \N$ such that if $G$ admits $(k,r)$-balanced separators, then $G$ admits a $(k',r')$-coverable tree decomposition. 
\end{conjecture}

Conversely, it is straightforward to show that every graph that admits a tree decomposition where every bag is $(k,r)$-coverable also admits $(k,r)$-balanced separators using a standard argument. Therefore, if proven true, Conjecture~\ref{conj: coarse bal sep and tw} would establish that the correspondence between balanced separators and treewidth is preserved under quasi-isometry.
The following strengthening of \cref{conj: coarse bal sep and tw} is also of interest.
\begin{conjecture} \label{conj: coarse bal sep and tw constant radius}
 For every $k,r \in \N$, there exists $k' \in \N$ such that if $G$ admits $(k,r)$-balanced separators, then $G$ admits a $(k',r)$-coverable tree decomposition. 
\end{conjecture}

Abrishami et al. reduced \cref{conj: coarse bal sep and tw} to the case $r=1$ \cite{ACKPPR25}, and
 proved a weakening of Conjecture~\ref{conj: coarse bal sep and tw}, allowing logarithmic dependence on the number of vertices in either  the number of balls or the radius.
Abrishami et.al. \cite{ACKPPR25} also confirmed Conjecture~\ref{conj: coarse bal sep and tw} in graph classes with bounded doubling dimension (which is the requirement that, for every $r \in \N$, every ball of radius $2r$ can be covered by a bounded number of balls of radius $r$). 
Chudnovsky and Hickingbotham proved Conjecture~\ref{conj: coarse bal sep and tw} for hereditary $K_{t,t}$-induced-subgraph-free graph classes when $r=1$.
\begin{theorem} [\cite{CH25Coarse}] \label{thm: r=1 coarse bal sep thm}
     For all $k,t \in \N$, there exist $k'$ such that the following holds: Let $G$ be $K_{t,t}$-induced-subgraph-free and $(k,1)$-separable. Then $G$ admits an internally $(k',2)$-coverable tree decomposition.
\end{theorem}

For a positive integer $t$, we denote by $\mathcal{C}_t$ the class of $K_{t,t}$-induced-minor-free graphs.
Our results strengthen \cref{thm: r=1 coarse bal sep thm} in multiple ways at the cost of bounding the clique number.
%
%
%
%
%
%
%
%
We prove that \cref{conj: coarse bal sep and tw constant radius} holds in separable subclasses of
$\mathcal{C}_t$ with bounded clique number.

\begin{restatable}{theorem}{mainthm}
\label{thm: general coarse bal sep thm}
    For every $r,t\in \mathbb{N}$, there exists a polynomial $p=p_{r,t}$ such that the following holds.
    For every $k,s\in \mathbb{N}$, if $G \in \class$ such that $\omega(G)<s$ and $G$ is $(k,r)$-separable then $G$ admits an internally $(p(ks),r)$-coverable tree decomposition.
\end{restatable}

This shows that \cref{conj: coarse bal sep and tw constant radius} holds for graphs in $\class$ with bounded clique number, while also providing additional control over the dependence of $k'$ on $k$ and the clique number. The proof of \cref{thm: r=1 coarse bal sep thm} relies on a result concerning the behavior of large neighborhoods within a given set in $K_{t,t}$-induced-subgraph-free graphs (Theorem 5.1 in \cite{treealpha5}). A substantial part of the present paper is devoted to establishing a distance-$r$ generalization of this result for graphs in $\class$:

\begin{restatable}{theorem}{theLemma}\label{theLemma}
    For all $t,r\in\N$, there exists a polynomial $f=f_{t,r}$ such that the following holds. 
    Let $C,s\in \N$ such that $C\geq 2$, let $G\in \class$ with $\omega(G)<s$, and let $Y \subseteq V(G)$. Then there is a subset $X \subseteq V(G)$ with $|X| \leq f(Cs)$ such that for every $v\in G\sm X$ we have that $$|N^{r}_{G\sm X}[v]\cap Y| \leq \frac{|Y|}{C}.$$
\end{restatable}

\cref{theLemma} will be used in a forthcoming paper by the first two authors, Lokshtanov, and E~S.

\subsection{Preliminaries and notation}\label{subsection notation}
We take  standard concepts and notation not explained here from~\cite{Diestel:GT5}. 
For a positive integer $k$, we use $[k]$ to denote the set $\{1, 2, \dots, k\}$.
Given a graph~$G$, we denote its vertex set by~$V(G)$ and its edge set by~$E(G)$. 
An \defn{independent set}, or \defn{stable set}, is a subset of vertices of $V(G)$ which are pairwise non-adjacent. 
A \defn{clique} in a graph $G$ is a set of pairwise adjacent vertices. A \defn{complete bipartite graph with bipartition $(X,Y)$} is a graph $G$ with vertex set $X\cup Y$ such that $xy\in E(G)$ if and only if $x\in X$ and $y\in Y$. For a positive integer $t$, let $K_t$ denote a clique on $t$ vertices and let $K_{t,t}$ denote a complete bipartite graph where $X$ and $Y$ contain exactly $t$ vertices each. 
The \defn{independence number} of a graph $G$, denoted by $\alpha(G)$, is the maximum size of an independent set in $G$. The \defn{clique number} of $G$, denoted by $\omega(G)$, is the maximum size of a clique in $G$.
For a set $X \subseteq V(G),$ we denote by $G[X]$ the subgraph of $G$ induced by $X$, and by $G \setminus X$ the subgraph of $G$ induced by $V(G) \setminus X$. In this paper, we  use induced subgraphs and their vertex sets interchangeably. For a graph $H$, we say that $H$ is an \defn{induced subgraph of $G$}
if there exists $X \subseteq V(G)$ such that $H$ is isomorphic to $G[X]$, and we
say that $H$ is  an \defn{induced minor of $G$} if $H$ can be obtained from $G$ by a sequence of vertex deletions and edge contractions (and by deleting parallel edges generated in the process).

A \defn{path} in a graph is an induced subgraph that is a path. We denote by $P=p_1 \dd \dots \dd p_k$
a path in $G$, where $p_ip_j \in E(G)$ if $|j-i|=1$. We say that $p_1$ and $p_k$ are
the \defn{ends} of $P$. The \defn{interior} of $P$, denoted by
$P^*$, is the set $P \setminus \{p_1,p_k\}$.
For $i,j \in \{1, \dots, k\}$ we denote by $p_i \dd P  \dd p_j$ the subpath of $P$ with ends $p_i,p_j$.
The \defn{length} of a path is the number of edges in it.
A \defn{$(u,v)$-path} is a path whose set of endpoints is equal to~$\{u,v\}$. 
Given a graph~$G$ and sets~${X,Y \subseteq V(G)}$ of vertices, an \defn{$(X,Y)$-path} is an $(x,y)$-path in~$G$ with~${x \in X}$, ${y \in Y}$, 
and no internal vertex in~${X \cup Y}$. 
The \defn{distance between two vertices} $u,v \in V(G)$, denoted by $\dist_G(u,v)$, is the length of the shortest $(u, v)$-path in $G$. The \defn{distance between two sets} $X$ and $Y$, denoted by $\dist_G(X,Y)$, is the minimum of the distance $\dist_G(x,y)$ among all $x\in X$ and $y \in Y$. For a vertex $v\in V(G)$, we use $N^r_G(v)$ and $N^r_G[v]$ to denote the sets of vertices at distance exactly $r$ and at most $r$ from $v$ in $G$, respectively. For a set $X \subseteq V(G)$, we let $N^r_G[X] = \bigcup_{v\in X}N^r_G[v]$ and $N^r_G(X) = N^r_G[X]\sm N^{r-1}_G[X]$. 
We say that a vertex $u$ is an \defn{$r$-neighbor} of $v$ if $u \in N^r_G(v)$.
If $r=1$ we may omit the superscript or the prefix. Similarly, if the ambient graph $G$ is clear from the context, we may omit it in the notation of both distances and neighborhoods.

\section{Large $r$-neighborhoods: the proof of \cref{theLemma} }

The following result due to Chudnovsky, Codsi, Lokshtanov, Milani\v{c}, and Sivashankar \cite{treealpha5} is an important step in the proof of Theorem~\ref{thm: r=1 coarse bal sep thm}. 
\begin{lemma}[Theorem 5.1 in \cite{treealpha5}] \label{lem: TI5 thm base case}
    Let $C, t\in \N$ with $C \geq 2$, and let $G$ be $ K_{t,t}$-induced-minor-free. Fix $Y \subseteq V(G)$ and define 
    $$Z= \left\{z\in V(G) \colon \alpha(N(z)\cap Y) \geq \frac{\alpha(Y)}{C}\right\}.$$
    Then  $\min(\alpha(Y), \alpha(Z)) \leq (512C)^{2t^{2t}}$.
\end{lemma}

The objective of this section is to prove \cref{theLemma}, which can be seen as a generalization of Lemma \ref{lem: TI5 thm base case} to distance-$r$ neighborhoods. This will be a crucial step in the proof of Theorem~\ref{thm: general coarse bal sep thm}, and is likely to  be of independent interest.

We start with the following:
\begin{theorem} \label{thm: TI5 generalization alpha}
     There exists a function $M(C, s,t,r)$ that is polynomial in the parameters $s$ and $C$ such that the following holds: Let $C,s,t, r\in \N$ and let $G\in\class$ such that $\omega(G)<s$ and $C\geq 2$. Let $Y \subseteq V(G)$. Then there is a subset $X \subseteq V(G)$ with $\alpha(X) \leq M(C,s,t,r)$ such that if
    $$Z= \left\{z\in V(G)\sm X\  \colon \alpha(N^{r}_{G\sm X}[z]\cap Y) \geq \frac{\alpha(Y)}{C}\right\}$$
    then $\min(\alpha(Y), \alpha(Z)) \leq M(C, s,t,r)$.
\end{theorem}

To prove this, the following results will be required. For positive integers $a,b$, let $R(a,b)$ be the smallest integer, often referred to as the \defn{Ramsey number}, such that every graph on $R(a,b)$ vertices contains either a clique of size $a$ or stable set of size $b$.

\begin{theorem}[\cite{ramsey}]\label{thm:ramsey}
    For all $a,b\in \N$, $R(a,b) \leq a^{b-1}$.
\end{theorem}

 A \defn{hypergraph} $H = (V(H),E(H))$ consists of a non-empty set of vertices $V(H)$ and a non-empty set of \defn{hyperedges} $E(H) \subseteq 2^{V(H)}$.  Given a hypergraph $H$, we define the following parameters. We denote by $\tau(H)$ the minimum size of a \defn{hitting set}, which is a set of vertices that meets all the hyperedges. We denote by $\nu(H)$ the maximum size of a \defn{matching}, which is a set of pairwise disjoint hyperedges. We denote by $\lambda(H)$ the size of the largest set $E$ of hyperedges such that for every $e,f \in E$, there exists $v\in V(H)$ such that $$v\in (e\cap f)\sm \bigcup_{\substack{g\in E\\ g\neq e,f}} g.$$

We will require the following result due to Ding, Seymour, and Winkler.

\begin{theorem}[Theorem 1.1 in \cite{DSW94}]\label{hypergraphs hitting packing}
For every hypergraph $H$ with no empty hyperedge, $$\tau(H)\leq 11\lambda(H)^2(\lambda(H)+\nu(H)+3)\binom{\lambda(H)+\nu(H)}{\nu(H)}^2.$$    
\end{theorem}

Let $G$ be a graph.
For $k\in \N$, a \defn{proper $k$-coloring} of $G$ is a mapping $c: V(G) \to [k]$ such that $c(u)\neq c(v)$ whenever $uv \in E(G)$. The \defn{chromatic number} of $G$, denoted by $\chi(G)$, is the minimum $k$ such that $G$ admits a proper $k$-coloring. Since every color class forms an independent set, we deduce the following simple observation.
\begin{observation} \label{obs: size bound by alpha and chromatic number}
    For every  graph $G$, $|V(G)| \leq \alpha(G) \cdot \chi(G)$. 
\end{observation}

A class of graphs $\cC$ is \defn{(polynomially) $\chi$-bounded} if there exists a (polynomial) function $p: \N \to \N$ such that $\chi(G) \leq p(\omega(G))$ for every $G \in \cC$. 
Bourneuf, Buci\'c, Cook, and Davies \cite{BBCD23} proved the following result. 
\begin{lemma}[Theorem 1.2 \cite{BBCD23}]\label{lem: poly chi bound}
   For every $t \in \N$, $\class$ is polynomially $\chi$-bounded.
\end{lemma}

Lastly, we require a Ramsey-type result of Lozin and Razgon \cite{LR22}.
\begin{lemma} [Lemma 2 in \cite{LR22}] \label{lem: disjointsubsets} 
For all $d,t\in \N$, there is a positive integer $K = K(d, t)$ such that if a graph $G$ contains a collection of $K$ pairwise disjoint subsets of vertices, each of size at most $d$ and with at least one edge between every two of them, then $G$ contains a (not necessarily induced) $K_{t,t}$-subgraph.
\end{lemma}

%


\begin{proof} [Proof of Theorem~\ref{thm: TI5 generalization alpha}]

  Let $K(\cdot,\cdot)$ be the constant given by Lemma~\ref{lem: disjointsubsets}, let $R(\cdot,\cdot)$ denote the Ramsey number, and let $c_t(\cdot)$ be the polynomial given by Lemma~\ref{lem: poly chi bound}. We define   

    \begin{align*}
         &R(y) = R(K(yt,t),t) ,  &\nu(x,y) = 2xc_t(s)^{y+1}(R(y)+1)\\
         &\tau(x,y) = 11(2t)^2(2t+\nu(x,y)+3)\binom{2t + \nu(x,y)}{\nu(x,y)}^2, \text{  } &\iota(x,y) = c_t(s)^{y+1}t\cdot\tau(x,y)^{R(y)}
    \end{align*}
   
    and \begin{equation*}
        M(C, s,t,r) = 
       \begin{cases}
           (512C)^{2t^{2t}} \quad\quad\quad\quad\quad\quad\quad\quad\quad\quad\quad\quad \text{ if } r=1\\
           2 M(2C\iota(C,r)(R(r)+1), s,t,r-1) \quad \text{ otherwise.}
       \end{cases}
    \end{equation*}

    \sta{\label{Polynomial in s} The function $M(C,s,t,r)$ is a polynomial in the parameters $s$ and $C$.}
     By construction, $M(C,s,t,1)$ is a polynomial the parameters $C$ and $s$. Suppose inductively that $M(C,s,t,r-1)$ is a polynomial in $s$ and $C$. Since $c_t(\cdot)$ is a polynomial, the function $\nu(C,r)$ is a polynomial in $s$ and $C$.
    Note that the binomial term in $\tau(C,r)$ is bounded by $\binom{2t+\nu(C,r)}{\nu(C,r)} \leq \frac{(2t+\nu(C,r))^{2t}}{(2t)!} $. Since $\tau(C,r)$ is a polynomial in $s$ and $C$, we conclude that $\iota(C,r)$ is a polynomial in $s$ and $C$. Since the composition of polynomials is a polynomial, we conclude that  $2 M(2C\iota(C,r)(R(r)+1), s,t,r-1)$ is a polynomial in the parameters $s$ and $C$. This proves \eqref{Polynomial in s}. \\

    We may assume that $\alpha(Y) > M(C,s,t, r)$ as otherwise there is nothing to show.
    We proceed by induction on $r$. When $r=1$, the result follows from Lemma~\ref{lem: TI5 thm base case} with $X = \emptyset$ and the observation that $\alpha(N[z] \cap Y) = \alpha(N(z) \cap Y)$, unless $N[z] \cap Y = z$ in which case $z\notin Z$ since $\frac{\alpha(Y)}{C} > 1$. 

    Now let $r \in \N$, and assume that the result holds for $1,\dots, r-1$.
    By definition of $M$, we have $\alpha(Y) > M(2C\iota(C,r) (R(r)+1),s,t, r-1)$.
    By the inductive hypothesis, there is a set $X_{r-1}$ with 
    $\alpha(X_{r-1}) \leq M(2C\iota(C,r)(R(r)+1),s,t, r-1) $ 
    for which 
    $$Z_{r-1} = \set{z\in V(G)\sm X_{r-1} : \alpha(N_{G\sm X_{r-1}}^{r-1}[z] \cap Y) \geq \frac{\alpha(Y)}{2\iota(C,r) C(R(r)+1)}}$$
    has $\alpha(Z_{r-1}) \leq M(2C\iota(C,r)(R(r)+1),s,t,r-1)$.
    Let $X = Z_{r-1} \cup X_{r-1}$. Then, since $r \geq 2$,
\begin{align*}
    \alpha(X) \leq 2 M(2C\iota(C,r)(R(r)+1),s,t,r-1) =  M(C,s,t,r).
\end{align*}

 In what follows, we abuse the notation by setting $\iota(C,r) = \iota$, $R = R(r)$, and $c=c_t(s)$.

    \sta{\label{the lemma claim cleaned r-1} For every $v\in G\sm X$, we have that $\alpha(N_{G\sm X}^{r-1}[v] \cap Y) < \frac{\alpha(Y)}{2\iota C(R+1)}$}
    
    \eqref{the lemma claim cleaned r-1} is a direct consequence of the definition of $X$.
    \\
    
    Let
    $$Z= \set{z\in V(G)\sm X : \alpha(N_{G \sm X}^{r}[z] \cap Y) \geq \frac{\alpha(Y)}{C}},$$
    we will show that $\alpha(Z) \leq M(C,s,t,r)$. For the sake of a contradiction, suppose not.  Let
    $I$ be  an  independent subset of $Z$ of size $\iota$ which exists since $M(C,s,t,r) \geq \iota$.
    Let $Y_{r-1} = N_{G\sm X}^{r-1}[I] \cap Y$.
    Set $G' = G\sm X$ and $Y ' = Y \sm (X\cup Y_{r-1})$. Then in $G'$ there is no path of length less than $r$ from $I \cap G'$ to $Y'$.

\sta{\label{claim that I still has paths to Y'} For every $z\in I$, we have that $\alpha(N_{G\sm X}^{r}(z) \cap Y') \geq  \frac{\alpha(Y)}{2 C}$}

It follows from \eqref{the lemma claim cleaned r-1} that 
    \begin{align*}
    \alpha (N_{G'}^{r}(z) \cap Y') 
    &\geq\alpha (N_{G\setminus X}^{r}[z] \cap Y)  -\alpha (Y_{r-1})  \\
    &>  \frac{\alpha(Y)}{C} - |I| \cdot \frac{\alpha(Y)}{2C\iota (R+1)}\\
     & \geq \frac{\alpha(Y)}{2C}.
    \end{align*}
  In particular, for every $z\in I$, there exist paths from $z$ to $Y'$ of length $r$. This proves \eqref{claim that I still has paths to Y'}.\\
   

Since $\class$ is hereditary and by \cref{lem: poly chi bound}, $\chi(G') \leq c$. 
Fix a proper $c$-coloring of $G'$. For a given path $p_0\dd p_1\dd \dots\dd p_r$ of length $r$ in $G'$, we call the sequence of colors $(c(p_0), c(p_1), \dots, c(p_r))$ the \defn{type} of the path. 
For every $z\in I$ and every $y\in  Y' \cap N_{G'}^r(z)$, let $P_{z,y}$ be a $(z,y)$-path of length $r$.
For every $z \in I$, let the \defn{type of $z$} be the most common type among the paths in the set $\set{P_{z,y} : y\in  Y' \cap N_{G'}^r(z) }$ (breaking ties arbitrarily). 
Let $T$ be the most common type among the vertices in $I$ (again breaking ties arbitrarily).
    Let $I'\subseteq I$ be the set of vertices of type $T$, and let $Y''$ be the set of vertices in $Y'$ that can be reached by paths of type $T$ from vertices in $I'$. 
    Let $G''$ be the graph induced by the union of the $(I',Y'')$-paths of type $T$. We define the \defn{0th level} as $L_0 = I'$ and the \defn{ith level} as $L_i = N^i_{G''}(I')$ for $i \in [r]$. For $v \in G''$, let $l(v)$ be the integer $i$ such that $v \in L_i$.

\sta{\label{I' properties} 
    $|I'| \geq \frac{1}{c^{r+1}} |I| \geq t\tau ^R$, and
    for every $z \in I'$,
    $\alpha( N_{G''}^r(z) \cap Y'') \geq \frac{1}{c^{r+1}} \cdot \frac{\alpha(Y)}{2C}.$}

The first assertion is immediate from the definition of $I'$ and since there are at most $c^{r+1}$ distinct possible types.
The second assertion follows from \eqref{claim that I still has paths to Y'} and the fact that $z$ is of type $T$.
This proves \eqref{I' properties}.

\hfill \break
\sta{\label{claim bipartite} $G''$ is bipartite and, therefore, $G''$ is $K_{t,t}$-subgraph-free.}
By construction, the levels are disjoint, and every edge has its ends in two consecutive levels (since otherwise we would be able to find a path of length less than $r$ between $I'$ and $Y''$). It is sufficient to show that each level $L_i$ is an independent set, as in this case, the set of even levels and the set of odd levels form a bipartition of $G''$.
Since $\dist(I',Y'')=r$ and every vertex in $G''$ is in an $(I',Y'')$-path of length $r$ of type $T$, it follows that for every $i$ and for every vertex $v\in L_i$, there exists an $(I',v)$-path of length $i$ that is an initial segment of a path of type $T$. Therefore, all the vertices of $L_i$ belong to the same color class, which proves the first assertion of \eqref{claim bipartite}. The second assertion follows since every bipartite graph in $\class$ is $K_{t,t}$-subgraph-free.
\\

    For $D \geq 0$, let  $\nu_D = \frac{2Cc^{r+1}(R+1)}{D+1}$ for any $D\geq 0$. Note that $\nu= \nu_0$.

    \begin{claim}\label{No small hitting sets} 
    For every $D$ such that $0\leq D \le R$, the following holds:
    Let $G^{\star}$ be an induced subgraph of $G''$, and define $I^{\star} = G^{\star}\cap I'$ and $Y^{\star}= G^{\star}\cap Y''$. Assume that $|I^{\star}| \geq t\tau^D$ and that for each $z\in I^{\star}$ we have $\alpha(N_{G^{\star}}^r(z)\cap Y^{\star}) \geq \frac{\alpha(Y)}{\nu_D}$. Then there exist $S\subseteq I^{\star}$ with $|S|=t$, and disjoint induced connected subgraphs $J_1,\dots, J_D \subseteq G^{\star}$, such that $|J_i|\leq tr $ and  $S\subseteq N_{G''}(J_i)$ for all $i \in [D]$.
    \end{claim}
    \begin{claimproof}
        We proceed by induction on $D$. When $D=0$, the statement is vacuously true. Now, assume inductively that the result holds for $1, \dots, D-1$. 

        By taking a subset if necessary, we may assume that $|I^\star|=t\tau^D$. Let $H$ be a hypergraph with vertex set $Y^{\star}$ and hyperedge set $I^{\star}$, where for every $z\in I^{\star}$, the hyperedge $e_z$ is defined as $\{y\in Y^{\star} : y \in N^r_{G^{\star}}(z)\}$. See \cref{fig: Hypergraph Construction}
        for a visualization.
        \begin{figure}[ht]
        \centering
        \scalebox{1}{\begin{tikzpicture}
	\begin{pgfonlayer}{nodelayer}
		\node [style=none] (0) at (-4.75, -2) {};
		\node [style=none] (1) at (-4.75, -4) {};
		\node [style=none] (2) at (4.25, -4) {};
		\node [style=none] (3) at (4.25, -2) {};
		\node [style=none] (4) at (-4.75, 4.25) {};
		\node [style=none] (5) at (-4.75, 2.25) {};
		\node [style=none] (6) at (4, 2.25) {};
		\node [style=none] (7) at (4, 4.25) {};
		\node [style=none] (8) at (6.25, 3.25) {$Y^\star$};
		\node [style=none] (9) at (6.25, -3) {$I^\star$};
		\node [style=vertex] (10) at (-2.25, -3) {};
		\node [style=vertex] (11) at (-3.75, 3.25) {};
		\node [style=vertex] (12) at (-2.5, 3.25) {};
		\node [style=vertex] (13) at (-1.25, 3.25) {};
		\node [style=none] (14) at (-3.25, 1) {};
		\node [style=none] (15) at (-2, 1) {};
		\node [style=none] (16) at (-1, -0.75) {};
		\node [style=none] (17) at (-4, 3.75) {};
		\node [style=none] (18) at (-4, 2.75) {};
		\node [style=none] (19) at (0, 2.75) {};
		\node [style=none] (20) at (0, 3.75) {};
		\node [style=none] (21) at (-5, 3.25) {$e_z$};
		\node [style=none] (22) at (-3, -3) {$z$};
		\node [style=vertex] (23) at (1.75, -3) {};
		\node [style=vertex] (24) at (1.5, 3.25) {};
		\node [style=vertex] (25) at (2.75, 3.25) {};
		\node [style=none] (26) at (2.25, 0.25) {};
		\node [style=none] (27) at (1.75, -3) {};
		\node [style=none] (28) at (0, 1) {};
		\node [style=none] (29) at (4, 3.25) {$e_{z'}$};
		\node [style=none] (30) at (2.5, -3) {$z'$};
		\node [style=none] (31) at (-1.25, 3.75) {};
		\node [style=none] (32) at (-1.25, 2.75) {};
		\node [style=none] (33) at (2.75, 2.75) {};
		\node [style=none] (34) at (2.75, 3.75) {};
		\node [style=vertex] (35) at (0, 3.25) {};
	\end{pgfonlayer}
	\begin{pgfonlayer}{edgelayer}
		\draw [style=simple] (3.center)
			 to (0.center)
			 to [in=-180, out=180, looseness=2.00] (1.center)
			 to (2.center)
			 to [in=360, out=0, looseness=2.00] cycle;
		\draw [style=simple] (5.center)
			 to (6.center)
			 to [in=360, out=0, looseness=2.00] (7.center)
			 to (4.center)
			 to [in=-180, out=180, looseness=2.00] cycle;
		\draw (10) to (16.center);
		\draw (16.center) to (15.center);
		\draw (14.center) to (10);
		\draw [style=grey fill] (20.center)
			 to (17.center)
			 to [in=-180, out=180, looseness=2.00] (18.center)
			 to (19.center)
			 to [in=360, out=0, looseness=2.00] cycle;
		\draw (26.center) to (27.center);
		\draw [style=grey fill] (31.center)
			 to [in=-180, out=180, looseness=2.00] (32.center)
			 to (33.center)
			 to [bend right=90, looseness=2.00] (34.center)
			 to cycle;
		\draw (24) to (26.center);
		\draw (28.center) to (27.center);
		\draw (28.center) to (13);
		\draw (16.center) to (28.center);
		\draw [style=simple] (11) to (14.center);
		\draw [style=simple] (12) to (15.center);
		\draw [style=simple] (25) to (26.center);
		\draw [style=simple] (15.center) to (13);
		\draw [style=simple] (35) to (28.center);
	\end{pgfonlayer}
\end{tikzpicture}}
        \caption{Construction of the hypergraph $H$.}
        \label{fig: Hypergraph Construction}
        \end{figure}
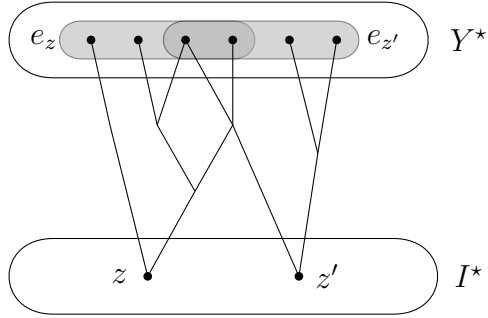
        The proof proceeds by using \cref{hypergraphs hitting packing} to find a vertex in $Y^\star$ that is contained in the $r$-neighborhood of many vertices of $I^\star$.
        We first show that $\nu(H)$ and $\lambda(H)$ are bounded by functions of $t$ and $\nu_D$. 
    
        \sta{\label{No large matching}  It holds that $\nu(H) \leq \nu_D$.}

        Suppose not, then there are more than $\nu_D$ pairwise disjoint subsets of $Y^{\star}$, each with at least $\frac{\alpha(Y)}{\nu_D}$ vertices. Therefore
        $$|Y^{\star}| > \frac{\alpha(Y)}{\nu_D} \cdot \nu_D = \alpha(Y).$$
       Since $Y^{\star}$ is an independent set contained in $Y$, this is a contradiction. This proves \eqref{No large matching}.

        \sta{\label{No large lambda}  It holds that $\lambda(H) < 2t$.} 
         Suppose that $\lambda(H) \geq 2t$.
        We will obtain a contradiction by constructing a $1$-subdivision of  $K_{2t}$, from which we obtain a $K_{t,t}$ induced minor. 
        Let $z_1,\dots,z_{2t}$ be vertices in $I'$ such that the hyperedges $e_{z_1},\dots,e_{z_{2t}}$ exhibit $\lambda(H)\geq 2t$. For every $i<j \in[2t]$, let $y_{i,j}$ be a vertex in  $(e_{z_i}\cap e_{z_j})\sm \bigcup_{k\neq i,j} e_{z_k}$. We call $y_{i,j}$ the \defn{private $r$-neighbor} of $z_i$ and $z_j$. Let $P_{i,j}$ and $P_{j,i}$ be paths of length $r$ from $z_i$ to $y_{i,j}$ and from $z_j$ to $y_{i,j}$, respectively, such that the union $|P_{i,j}\cup P_{j,i}|$ is minimal. Let $s_{i,j}$ be the vertex in $P_{i,j}\cap P_{j,i}$ with $l(s_{i,j})$ minimum. Such a vertex exists because $y_{i,j} \in P_{i,j}\cap P_{j,i}$. 
        We define $P_{i,j}' = z_i \dd P_{i,j} \dd s_{i,j} \sm s_{i,j}$ and $R_{i} = \bigcup_{j\neq i}P_{i,j}'$. 
    
        \sta{If $i\neq j$, then $R_i$ and $R_j$ are disjoint and anticomplete.\label{claim: Ri anticomplete}}
        Suppose not, then since $|R_i \cup R_j|\geq 2$ and since $R_i$ and $R_j$ are each connected subgraphs, 
        there is an edge $e = u_iu_j$ with $u_i\in P_{i,k}' \subseteq R_i$ and $u_j \in P_{j,l}' \subseteq R_j$ for some $k,l$. See Figure~\ref{fig: Ri and Rj disjoint dotted} for a visualization. 

        Since for every edge $e$ the ends of $e$ belong to two consecutive levels, without loss of generality, $u_i \in L_{m}$ and $u_j\in L_{m-1}$ for some $m \in [r]$.
        Now, $z_j \dd P_{j,l}\dd u_j \dd u_i\dd P_{i,k}\dd y_{i,k} $ is a path of length $r$ from $z_j$ to $y_{i,k}$. Since $y_{i,k}$ is the private $r$-neighbor of $z_i$ and $z_k$, it follows that $k=j$.
        Then $u_i$ is a vertex in the intersection of a $(z_j,y_{i,j})$-path and $P_{i,j}$. Since $u_i \in P_{i,j}'$,  $l(u_i) <l(s_{i,j})$,
        contradicting the minimality of $s_{i,j}$. This proves \eqref{claim: Ri anticomplete}.
    
        \begin{figure}[ht]
        \centering
        \scalebox{1}{\begin{tikzpicture}
	\begin{pgfonlayer}{nodelayer}
		\node [style=none] (27) at (-2.75, 5.25) {};
		\node [style=none] (28) at (-2.75, 5.25) {$y_{i,k}$};
		\node [style=vertex] (29) at (-2, 5.25) {};
		\node [style=none] (30) at (2.75, 5.25) {$y_{j,l}$};
		\node [style=none] (31) at (-1.25, 1.25) {$u_i$};
		\node [style=none] (32) at (0.75, -2.5) {$z_j$};
		\node [style=none] (33) at (-1, -2.5) {$z_i$};
		\node [style=none] (34) at (-4, -2.5) {$z_k$};
		\node [style=none] (35) at (-2.75, 2.5) {$s_{i,k}$};
		\node [style=vertex] (36) at (-1.5, -2.25) {};
		\node [style=vertex] (37) at (-3.5, -2.25) {};
		\node [style=vertex] (38) at (-2, 2.5) {};
		\node [style=vertex] (39) at (2, 5.25) {};
		\node [style=vertex] (40) at (2, 2.5) {};
		\node [style=vertex] (41) at (1.25, -2.25) {};
		\node [style=tiny dot] (42) at (-2.475, 1) {};
		\node [style=tiny dot] (44) at (-1.65, -0.75) {};
		\node [style=vertex] (45) at (-1.825, 1) {};
		\node [style=tiny dot] (47) at (-3.025, -0.75) {};
		\node [style=vertex] (48) at (1.45, -1) {};
		\node [style=tiny dot] (49) at (2, 3.75) {};
		\node [style=tiny dot] (50) at (-2, 3.75) {};
		\node [style=tiny dot] (51) at (3.075, -0.875) {};
		\node [style=tiny dot] (52) at (2.475, 1) {};
		\node [style=tiny dot] (54) at (1.775, 1) {};
		\node [style=none] (69) at (-3.5, 5.9) {};
		\node [style=none] (70) at (-3.5, 4.65) {};
		\node [style=none] (71) at (3.5, 4.65) {};
		\node [style=none] (72) at (3.5, 5.9) {};
		\node [style=none] (73) at (-4, -1.75) {};
		\node [style=none] (74) at (-4, -3) {};
		\node [style=none] (75) at (4, -3) {};
		\node [style=none] (76) at (4, -1.75) {};
		\node [style=none] (77) at (5.25, 5.35) {$Y^\star$};
		\node [style=none] (78) at (5.75, -2.375) {$I^\star$};
		\node [style=vertex] (79) at (3.5, -2.25) {};
		\node [style=none] (80) at (4, -2.5) {$z_l$};
		\node [style=none] (81) at (1.175, -0.375) {$u_j$};
		\node [style=none] (82) at (2.75, 2.5) {$s_{j,l}$};
		\node [style=none] (86) at (-0.75, -1) {$P_{i,k}'$};
		\node [style=none] (87) at (1, 1.5) {$P_{j,l}'$};
	\end{pgfonlayer}
	\begin{pgfonlayer}{edgelayer}
		\draw [style=simple, in=90, out=-90, looseness=1.25] (29) to (38);
		\draw [style=simple] (38) to (36);
		\draw [style=simple] (40) to (41);
		\draw [style=simple] (37) to (38);
		\draw [style=simple] (40) to (39);
		\draw [style=simple] (71.center)
			 to (70.center)
			 to [bend left=90, looseness=2.50] (69.center)
			 to (72.center)
			 to [bend left=90, looseness=2.50] cycle;
		\draw [style=simple] (75.center) to (74.center);
		\draw [style=simple, bend left=90, looseness=2.50] (74.center) to (73.center);
		\draw [style=simple] (73.center) to (76.center);
		\draw [style=simple, bend left=90, looseness=2.50] (76.center) to (75.center);
		\draw [style=simple] (40) to (79);
		\draw [style=simple] (45) to (48);
	\end{pgfonlayer}
\end{tikzpicture}}
        \caption{Visualization of the proof of \eqref{claim: Ri anticomplete}.}
        \label{fig: Ri and Rj disjoint dotted}
        \end{figure}
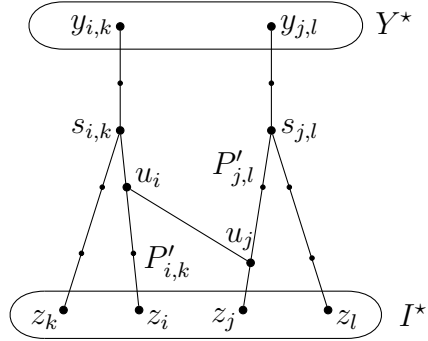

        \sta{If $k\notin \set{i,j}$, then for every $l$, $P_{k,l}$ is anticomplete to $s_{i,j}$. \label{claim: sij does not see stuff}}
    
       Suppose that $s_{i,j}$ is adjacent to a vertex $u\in P_{k,l}$. 
       Then $|l(u)-l(s_{i,j})| =1$.
        First, suppose that $s_{i,j} \in L_{m-1}$ and $u \in L_m$ for some $m \in [r]$. See the left-hand side of Figure~\ref{fig: Dotted path Pkl not adjacent to sij}. Then $z_i\dd P_{i,j} \dd s_{i,j}$ and $z_j\dd P_{j,i} \dd s_{i,j}$ are paths of length $m-1$ and $u\dd P_{k,l} \dd y_{k,l}$ is a path of length $r-m$. Since $s_{i,j}$ and $u$ are adjacent, we have found a $(z_i, y_{k,l})$-path and a $(z_j,y_{k,l})$-path of length $r$. Since $y_{k,l}$ is the private $r$-neighbor of $z_k$ and $z_l$, we conclude that $k\in\set{i,j}$.
    
        Now suppose $s_{i,j} \in L_{m}$ and $u \in L_{m-1}$. See the right-hand side of Figure~\ref{fig: Dotted path Pkl not adjacent to sij}. Then  
        $z_k \dd P_{k,l} \dd u$ is a path of length $m-1$ and $s_{i,j}\dd P_{i,j} \dd y_{i,j}$ is a path of length $r-m$. Since $u$ and $s_{i,j}$ are adjacent, $z_k \dd P_{k,l} \dd u \dd s_{i,j}\dd P_{i,j} \dd y_{i,j}$ is a path of length $r$ from $z_k$ to $y_{i,j}$, which implies that $y_{i,j} \in e_{z_k}$. However, $y_{i,j}$ is the private $r$-neighbor of $z_i$ and $z_j$, so we conclude that $k\in\set{i,j}$. This proves \eqref{claim: sij does not see stuff}.
        \\
    
        \begin{figure}[ht]
        \centering
        \scalebox{1}{\begin{tikzpicture}
	\begin{pgfonlayer}{nodelayer}
		\node [style=none] (1) at (-7.5, 4.25) {$y_{i,j}$};
		\node [style=vertex] (2) at (-6.75, 4.25) {};
		\node [style=none] (3) at (-2.5, 4.25) {$y_{k,l}$};
		\node [style=none] (4) at (-2.75, 2.25) {$u$};
		\node [style=none] (5) at (-2.75, -3) {$z_k$};
		\node [style=none] (6) at (-5.25, -3) {$z_j$};
		\node [style=none] (7) at (-8.25, -3) {$z_i$};
		\node [style=none] (8) at (-7.5, 0.25) {$s_{i,j}$};
		\node [style=vertex] (9) at (-5.75, -2.75) {};
		\node [style=vertex] (10) at (-7.75, -2.75) {};
		\node [style=vertex] (11) at (-6.75, 0.25) {};
		\node [style=vertex] (12) at (-3.25, 4.25) {};
		\node [style=vertex] (13) at (-3.25, 2.25) {};
		\node [style=vertex] (14) at (-3.25, -2.75) {};
		\node [style=tiny dot] (15) at (-7.425, -1.75) {};
		\node [style=tiny dot] (16) at (-7.1, -0.75) {};
		\node [style=tiny dot] (17) at (-6.075, -1.75) {};
		\node [style=tiny dot] (18) at (-6.425, -0.75) {};
		\node [style=tiny dot] (19) at (-3.25, -1.25) {};
		\node [style=tiny dot] (20) at (-3.25, -0.25) {};
		\node [style=tiny dot] (21) at (-3.25, 0.75) {};
		\node [style=tiny dot] (23) at (-6.75, 3.25) {};
		\node [style=tiny dot] (24) at (-6.75, 2.25) {};
		\node [style=tiny dot] (25) at (-6.75, 1.25) {};
		\node [style=tiny dot] (26) at (-3.25, 3.25) {};
		\node [style=none] (27) at (3.25, 4.25) {};
		\node [style=none] (28) at (3.25, 4.25) {$y_{i,j}$};
		\node [style=vertex] (29) at (4, 4.25) {};
		\node [style=none] (30) at (8.25, 4.25) {$y_{k,l}$};
		\node [style=none] (31) at (8, 0) {$u$};
		\node [style=none] (32) at (8, -2.75) {$z_k$};
		\node [style=none] (33) at (5.5, -2.75) {$z_j$};
		\node [style=none] (34) at (2.5, -2.75) {$z_i$};
		\node [style=none] (35) at (3.25, 1.5) {$s_{i,j}$};
		\node [style=vertex] (36) at (5, -2.75) {};
		\node [style=vertex] (37) at (3, -2.75) {};
		\node [style=vertex] (38) at (4, 2) {};
		\node [style=vertex] (39) at (7.5, 4.25) {};
		\node [style=vertex] (40) at (7.5, 0) {};
		\node [style=vertex] (41) at (7.5, -2.75) {};
		\node [style=tiny dot] (42) at (3.425, -0.75) {};
		\node [style=tiny dot] (43) at (3.625, 0.25) {};
		\node [style=tiny dot] (44) at (4.575, -0.75) {};
		\node [style=tiny dot] (45) at (4.375, 0.25) {};
		\node [style=tiny dot] (46) at (7.5, -1.75) {};
		\node [style=tiny dot] (47) at (3.2, -1.75) {};
		\node [style=tiny dot] (48) at (7.5, -1) {};
		\node [style=tiny dot] (49) at (7.5, 3.25) {};
		\node [style=tiny dot] (50) at (4, 3) {};
		\node [style=tiny dot] (51) at (7.5, 1.25) {};
		\node [style=tiny dot] (52) at (7.5, 2.25) {};
		\node [style=tiny dot] (54) at (4.8, -1.75) {};
		\node [style=none] (55) at (-8.5, 4.9) {};
		\node [style=none] (56) at (-8.5, 3.65) {};
		\node [style=none] (57) at (-1.5, 3.65) {};
		\node [style=none] (58) at (-1.5, 4.9) {};
		\node [style=none] (59) at (-8.5, -2.25) {};
		\node [style=none] (60) at (-8.5, -3.5) {};
		\node [style=none] (61) at (-1.5, -3.5) {};
		\node [style=none] (62) at (-1.5, -2.25) {};
		\node [style=none] (64) at (-10.25, 4.425) {$Y^\star$};
		\node [style=none] (68) at (-10.25, -2.825) {$I^\star$};
		\node [style=none] (69) at (2.5, 4.9) {};
		\node [style=none] (70) at (2.5, 3.65) {};
		\node [style=none] (71) at (9.5, 3.65) {};
		\node [style=none] (72) at (9.5, 4.9) {};
		\node [style=none] (73) at (2.5, -2.25) {};
		\node [style=none] (74) at (2.5, -3.5) {};
		\node [style=none] (75) at (9.5, -3.5) {};
		\node [style=none] (76) at (9.5, -2.25) {};
		\node [style=none] (77) at (11.25, 4.375) {$Y^\star$};
		\node [style=none] (78) at (11.25, -2.875) {$I^\star$};
		\node [style=none] (80) at (-2, -0.5) {$P_{k,l}$};
		\node [style=none] (81) at (9, 2) {$P_{k,l}$};
		\node [style=none] (82) at (2.25, -0.5) {$P_{i,j}$};
		\node [style=none] (83) at (-8, 2) {$P_{i,j}$};
	\end{pgfonlayer}
	\begin{pgfonlayer}{edgelayer}
		\draw [style=simple, in=90, out=-90, looseness=1.25] (2) to (11);
		\draw [style=simple] (11) to (9);
		\draw [style=simple] (13) to (14);
		\draw [style=simple] (10) to (11);
		\draw [style=simple] (11) to (13);
		\draw [style=simple] (13) to (12);
		\draw [in=90, out=-90, looseness=1.25] (29) to (38);
		\draw [style=simple] (38) to (36);
		\draw (40) to (41);
		\draw [style=simple] (37) to (38);
		\draw (38) to (40);
		\draw [style=simple] (40) to (39);
		\draw [style=simple] (57.center)
			 to (56.center)
			 to [bend left=90, looseness=2.50] (55.center)
			 to (58.center)
			 to [bend left=90, looseness=2.50] cycle;
		\draw [style=simple] (61.center)
			 to (60.center)
			 to [bend left=90, looseness=2.50] (59.center)
			 to (62.center)
			 to [bend left=90, looseness=2.50] cycle;
		\draw [style=simple] (71.center)
			 to (70.center)
			 to [bend left=90, looseness=2.50] (69.center)
			 to (72.center)
			 to [bend left=90, looseness=2.50] cycle;
		\draw [style=simple] (75.center) to (74.center);
		\draw [style=simple, bend left=90, looseness=2.50] (74.center) to (73.center);
		\draw [style=simple] (73.center) to (76.center);
		\draw [style=simple, bend left=90, looseness=2.50] (76.center) to (75.center);
	\end{pgfonlayer}
\end{tikzpicture}}
        \caption{Visualization of the proof of \eqref{claim: sij does not see stuff}.}
        \label{fig: Dotted path Pkl not adjacent to sij}
    \end{figure}

    We now complete the proof of \eqref{No large lambda}.
        Consider the graph $\Gamma$ obtained by deleting every vertex that is not in $(\bigcup_{i\in [2t]} R_i) \cup (\bigcup_{i<j \in[2t]}s_{i,j})$ and contracting each $R_i$ to a single vertex $r_i$. Then $\Gamma$ is an induced minor of $G$. 
        By \eqref{claim: Ri anticomplete}, the set $\set{r_i}_{i\leq2t}$ is stable.
        In addition, \eqref{claim: sij does not see stuff} implies that if $k\notin \set{i,j}$, then $s_{i,j}$ is not adjacent to any vertex in $R_k$. Therefore, in $\Gamma$, $s_{i,j}$ is adjacent to $r_k$ if and only if $k\in \{i,j\}$.  
        It also follows from \eqref{claim: sij does not see stuff} that if $\set{k,l} \neq \set{i,j}$, then $s_{i,j}$ and $s_{k,l}$ are not adjacent, so the set $\bigcup_{i<j \in[2t]}s_{i,j}$ is stable. 
        Therefore, $\Gamma$ is a $1$-subdivision of a $K_{2t}$, which contains a $K_{t,t}$ induced minor, a contradiction. This completes the  proof of  \eqref{No large lambda}.\\

        By \eqref{No large matching}, $\nu(H) \leq \nu_D$ and by \eqref{No large lambda}, $\lambda (H)\leq 2t$. Note that $\nu_D \leq \nu$. By \cref{hypergraphs hitting packing}, we have
        \begin{equation*}
            \tau(H) \leq 11(2t)^2\left(2t+ \nu_D +3 \right)\binom{2t + \nu_D}{\nu_D}^2  \leq  \tau.
        \end{equation*}
        Let $\mathcal{T} \subseteq Y^{\star}$ be a hitting set of minimum size. Since $|\mathcal{T}| \leq \tau$, there exists a vertex $x \in \mathcal{T}$ such that the set 
        $\tilde{I}= \set{z\in I^{\star} : x\in e_z}$ has size $|\tilde{I}| \geq \frac{1}{\tau}|I^{\star}|$.
        For each $z\in \tilde{I}$, fix a $(z,x)$-path $Q_{z,x}$ of length $r$. Then set $\tilde{J}_D = \bigcup_{z\in \tilde{I}\ }(Q_{z,x} \sm {z} ) $. 
        
        Let $\tilde{G} = \left(G^{\star}\sm\left(\tilde{J}_D\cup I^{\star}\right)\right)\cup \tilde{I}$. We now check that, after deleting $\tilde{J}_D$, every vertex in $\tilde{I}$ still reaches a sufficiently large proportion of $Y^\star$ via paths of length $r$, which will allow us to apply induction.

        \sta{\label{inductive z neighborhood in Y} For every vertex $z\in \tilde{I}$, we have $|N^{r}_{\tilde{G}}(z) \cap Y^\star| \geq  \frac{\alpha (Y)}{\nu_{D-1}}$.}
        Let  $z\in \tilde{I}$ and define $\tilde{Y}_z = (N_{G^{\star}}^{r}(z) \cap Y^\star) \sm (N_{G^{\star}}^{r-1}(\tilde{J}_D \cap L_1) \cap Y^\star)$. We show that $\tilde{Y}_z\subseteq N^{r}_{\tilde{G}}(z) \cap Y^\star$.
        Let $y \in \tilde{Y}_z$, then there is a $(z,y)$-path $P_{z,y}$ in $G^{\star}$ of length $r$. We show that every such path is vertex-disjoint from $\tilde{J}_D$.
\begin{figure}[ht]
        \centering
        \scalebox{1}{\begin{tikzpicture}
	\begin{pgfonlayer}{nodelayer}
		\node [style=none] (35) at (-4.25, 5.45) {};
		\node [style=none] (36) at (-4.25, 4.55) {};
		\node [style=none] (37) at (-1.75, 5.45) {};
		\node [style=none] (38) at (-1.75, 4.55) {};
		\node [style=none] (0) at (-4.25, 5.75) {};
		\node [style=none] (1) at (-4.25, 4.25) {};
		\node [style=none] (2) at (5.25, 5.75) {};
		\node [style=none] (3) at (5.25, 4.25) {};
		\node [style=none] (4) at (-4.25, -3.25) {};
		\node [style=none] (5) at (-4.25, -5) {};
		\node [style=none] (6) at (5.25, -3.25) {};
		\node [style=none] (7) at (5.25, -5) {};
		\node [style=none] (8) at (7.75, 5) {$Y^\star \subseteq L_r$};
		\node [style=none] (9) at (7.75, -4) {$I^{\star} \subseteq L_0$};
		\node [style=none] (10) at (-2.5, -3.575) {};
		\node [style=none] (11) at (-2.5, -4.625) {};
		\node [style=none] (12) at (3.5, -3.575) {};
		\node [style=none] (13) at (3.5, -4.625) {};
		\node [style=vertex] (15) at (2.75, -4.075) {};
		\node [style=vertex] (16) at (0.75, -4.075) {};
		\node [style=vertex] (17) at (-1, -4.075) {};
		\node [style=tiny dot] (23) at (-4, 5) {};
		\node [style=tiny dot] (24) at (-3, 5) {};
		\node [style=tiny dot] (25) at (-2, 5) {};
		\node [style=vertex] (28) at (0.75, 5) {};
		\node [style=none] (31) at (-0.475, 3.675) {$\tilde{J}_D$};
		\node [style=none] (39) at (-4.25, -0.5) {};
		\node [style=none] (40) at (-4.25, -2.25) {};
		\node [style=none] (41) at (5.25, -0.5) {};
		\node [style=none] (42) at (5.25, -2.25) {};
		\node [style=none] (47) at (-2, -4.075) {$\tilde{I}$};
		\node [style=vertex] (54) at (1.775, -1.375) {};
		\node [style=vertex] (55) at (-1.275, -1.35) {};
		\node [style=vertex] (56) at (0.25, -1.35) {};
		\node [style=none] (57) at (7.75, -1.25) {$L_1$};
		\node [style=none] (58) at (-4.25, -0.85) {};
		\node [style=none] (59) at (-4.25, -1.9) {};
		\node [style=none] (60) at (2, -0.85) {};
		\node [style=none] (61) at (2, -1.9) {};
		\node [style=none] (64) at (-3.25, -1.35) {$\tilde{J}_D \cap L_1$};
		\node [style=none] (66) at (1.25, -4) {$z'$};
		\node [style=none] (67) at (3.25, -4.1) {$z$};
		\node [style=none] (69) at (1, 3) {$u$};
		\node [style=none] (71) at (-2, 0.75) {};
		\node [style=none] (72) at (-2.25, 6.325) {$N^{r-1}_{G\star}(\tilde{J}_D \cap L_1)$};
		\node [style=none] (73) at (4, 5) {};
		\node [style=none] (74) at (4, 5) {$y$};
		\node [style=vertex] (75) at (3.5, 5) {};
		\node [style=vertex] (76) at (3.5, 5) {};
		\node [style=vertex] (77) at (0.75, 2.25) {};
		\node [style=none] (78) at (0.75, 1.5) {};
		\node [style=none] (79) at (3.5, 0.75) {};
		\node [style=none] (80) at (3.5, 0.75) {$P_{z,y}$};
		\node [style=none] (81) at (1.25, 5) {$x$};
		\node [style=none] (82) at (-4.625, 6.25) {};
		\node [style=none] (83) at (-4.425, 5) {};
		\node [style=none] (84) at (0.775, -1.375) {$v$};
	\end{pgfonlayer}
	\begin{pgfonlayer}{edgelayer}
		\draw [style=simple] (0.center)
			 to (2.center)
			 to [bend left=90, looseness=2.25] (3.center)
			 to [in=360, out=180] (1.center)
			 to [bend left=90, looseness=2.25] cycle;
		\draw [style=simple] (4.center)
			 to (6.center)
			 to [bend left=90, looseness=2.25] (7.center)
			 to (5.center)
			 to [bend left=90, looseness=2.25] cycle;
		\draw [style=grey fill] (12.center)
			 to [bend left=90, looseness=2.25] (13.center)
			 to (11.center)
			 to [bend left=90, looseness=2.25] (10.center)
			 to cycle;
		\draw [style=grey fill] (35.center)
			 to (37.center)
			 to [bend left=90, looseness=2.25] (38.center)
			 to (36.center)
			 to [bend left=90, looseness=2.25] cycle;
		\draw [style=simple] (41.center)
			 to [bend left=90, looseness=2.25] (42.center)
			 to (40.center)
			 to [bend left=90, looseness=2.25] (39.center)
			 to cycle;
		\draw [style=grey fill] (61.center)
			 to (59.center)
			 to [bend left=90, looseness=2.25] (58.center)
			 to (60.center)
			 to [bend left=90, looseness=2.25] cycle;
		\draw [style=gray edge] (71.center) to (55);
		\draw [style=gray edge] (71.center) to (23);
		\draw [style=gray edge] (56) to (24);
		\draw [style=gray edge] (25) to (54);
		\draw [style=simple] (17) to (55);
		\draw [style=simple] (55) to (28);
		\draw [style=simple] (56) to (16);
		\draw [style=simple] (28) to (78.center);
		\draw [style=simple] (78.center) to (56);
		\draw [style=simple] (78.center) to (54);
		\draw [style=simple] (54) to (15);
		\draw [style=simple, bend left=45] (77) to (15);
		\draw [style=simple] (76) to (77);
		\draw [style=diredge, bend right=60] (82.center) to (83.center);
	\end{pgfonlayer}
\end{tikzpicture}}
        \caption{Visualization of the proof of \eqref{inductive z neighborhood in Y}.}
        \label{fig: Vertices have large neighborhoods}
\end{figure}
        Suppose, for contradiction, that there exists $u\in P_{z,y} \cap \tilde{J}_D$ with $l(u)=m$ for some $m$.
        Since $u \in \tilde{J}_D$, there exist  $z'\in \tilde{I}$ (not necessarily distinct from $z$), and a path $P_{z',x} \subseteq \tilde{J}_D$  such that $u \in P_{z',x}$. Let $v$ denote the vertex in $P_{z',x} \cap L_1$. See Figure ~\ref{fig: Vertices have large neighborhoods} for a visualization. Then $v\dd P_{z',x}\dd u\dd P_{z,y}\dd y$ is a path of length $r-1$ from $v \in \tilde{J}_D \cap L_1$ to $y$. This implies $y \in N^{r-1}_{G^\star}(\tilde{J}_D \cap L_1)\cap Y^\star$, contrary to the assumption.

        Therefore, the vertices in $\tilde{Y}_z$ can only be reached from $z$ via paths of length $r$ that are vertex-disjoint from $\tilde{J}_D$, which means that these paths are entirely contained in $\tilde{G}$.
        We conclude that $\tilde{Y}_z \subseteq N_{\tilde{G}}^r(z) \cap Y^\star$ and 
        \begin{align*}
            |N_{\tilde{G}}^r(z) \cap Y^\star| &\geq |N_{G^{\star}}^{r}(z) \cap Y^\star| - |N_{G^{\star}}^{r-1}(\tilde{J}_D \cap L_1) \cap Y^\star| \\
            & \geq \frac{\alpha(Y)}{\nu_D} - |I^{\star}| \cdot \frac{\alpha(Y)}{2\iota C(R+1)}\\
            &\geq \frac{\alpha(Y) (D+1)}{2c^{r+1}C(R+1)} - \frac{\alpha(Y)}{2c^{r+1}C(R+1)}  \\
            &\geq \frac{\alpha(Y)}{\nu_{D-1}},
        \end{align*}
        where the second inequality follows from \eqref{the lemma claim cleaned r-1} and the third inequality uses the fact 
        $\frac{|I^\star|}{\iota} \leq c^{-(r+1)}$.
        This proves \eqref{inductive z neighborhood in Y}.\\

    Observe that $\tilde{G}$ is an induced subgraph of $G^{\star}$ where $\tilde{G} \cap I' = \tilde{I}$ and $\tilde{G} \cap Y'' = Y^{\star} \sm \{x\}$. 
    By construction $|\tilde{I}| \geq \frac{1}{\tau} |I^{\star}| \geq t\tau^{D-1}$ and by \eqref{inductive z neighborhood in Y} for each $z\in \tilde{I}$ we have $\alpha(N_{\tilde{G}}^{r}(z) \cap (Y^{\star} \sm \{x\})) \geq \frac{\alpha(Y)}{\nu_{D-1}}$. Hence, by the inductive hypothesis, there exists  a set $S \subseteq \tilde{I}$ and disjoint induced subgraphs $J_1, \dots, J_{D-1}$ of $\tilde{G}$ such that 
    \begin{itemize}
        \item $|S| = t$,
        \item $|J_i| \leq r t$ for all $i\in [D-1]$, 
        \item  and $S \subseteq N_{G''}(J_i)$.
    \end{itemize}
    
    Recalling that $S \subseteq I^{\star}$, we let $J_D$ be the subgraph of $\tilde{J}_D$ induced by $\bigcup_{z\in S\ } (Q_{z,x}\sm \{z\})$. Now $|J_D| \leq rt$ and $S \subseteq N_{G''}(J_D)$. Moreover,  for each $i\leq D-1$, $J_i$ is vertex-disjoint from $J_D$ because it is an induced subgraph of $\tilde{G}$.
    We conclude that the set $S$ and $J_1, \dots, J_D$ satisfy the required properties and Claim~\ref{No small hitting sets} is proven.     
    \end{claimproof}

\begin{figure}[ht]
        \centering
        \scalebox{1}{\begin{tikzpicture}
	\begin{pgfonlayer}{nodelayer}
		\node [style=none] (0) at (-3.25, 4.75) {};
		\node [style=none] (1) at (-3.25, 2.75) {};
		\node [style=none] (2) at (4.75, 4.75) {};
		\node [style=none] (3) at (4.75, 2.75) {};
		\node [style=none] (4) at (-3.25, -2.25) {};
		\node [style=none] (5) at (-3.25, -4.25) {};
		\node [style=none] (6) at (4.75, -2.25) {};
		\node [style=none] (7) at (4.75, -4.25) {};
		\node [style=none] (8) at (7.25, 3.75) {$Y''$};
		\node [style=none] (9) at (7.5, -3.25) {$I'$};
		\node [style=none] (10) at (-2.75, -2.75) {};
		\node [style=none] (11) at (-2.75, -3.75) {};
		\node [style=none] (12) at (1.5, -2.75) {};
		\node [style=none] (13) at (1.5, -3.75) {};
		\node [style=none] (14) at (-4, -3.25) {S};
		\node [style=vertex] (15) at (-2.5, -3.25) {};
		\node [style=vertex] (16) at (-1.25, -3.25) {};
		\node [style=vertex] (17) at (1.25, -3.25) {};
		\node [style=tiny dot] (18) at (-0.5, -3.25) {};
		\node [style=tiny dot] (19) at (0, -3.25) {};
		\node [style=tiny dot] (20) at (0.5, -3.25) {};
		\node [style=vertex] (21) at (-0.75, 3.75) {};
		\node [style=vertex] (22) at (3.75, 3.75) {};
		\node [style=tiny dot] (23) at (0.75, 3.75) {};
		\node [style=tiny dot] (24) at (1.25, 3.75) {};
		\node [style=tiny dot] (25) at (1.75, 3.75) {};
		\node [style=none] (26) at (2.5, 0.75) {};
		\node [style=none] (27) at (-2.5, 0.75) {};
		\node [style=vertex] (28) at (-2.75, 3.75) {};
		\node [style=none] (29) at (4, 1.75) {$J_R$};
		\node [style=none] (30) at (0.75, 1.75) {$J_2$};
		\node [style=none] (31) at (-3, 1.75) {$J_1$};
		\node [style=vertex] (32) at (3, -3.25) {};
		\node [style=vertex] (33) at (4, -3.25) {};
		\node [style=vertex] (34) at (5, -3.25) {};
	\end{pgfonlayer}
	\begin{pgfonlayer}{edgelayer}
		\draw [style=simple] (0.center)
			 to (2.center)
			 to [bend left=90, looseness=2.25] (3.center)
			 to (1.center)
			 to [bend left=90, looseness=2.25] cycle;
		\draw [style=simple] (4.center)
			 to (6.center)
			 to [bend left=90, looseness=2.25] (7.center)
			 to (5.center)
			 to [bend left=90, looseness=2.25] cycle;
		\draw [style=grey fill] (12.center)
			 to [bend left=90, looseness=2.25] (13.center)
			 to (11.center)
			 to [bend left=90, looseness=2.25] (10.center)
			 to cycle;
		\draw [style=simple] (28) to (27.center);
		\draw [style=simple] (27.center) to (15);
		\draw [style=simple] (27.center) to (16);
		\draw [style=gray edge] (21) to (16);
		\draw [style=light grey] (22) to (26.center);
		\draw [style=light grey] (26.center) to (17);
		\draw [style=light grey] (16) to (26.center);
		\draw [style=gray edge] (21) to (15);
		\draw [style=gray edge] (21) to (17);
		\draw [style=light grey] (22) to (15);
		\draw [style=simple] (28) to (17);
	\end{pgfonlayer}
\end{tikzpicture}}
        \caption{Visualization of the induced subgraphs $J_1,\dots, J_R$ given by Claim~\ref{No small hitting sets}.}
        \label{fig: Disjoint Induced Subgraphs}
\end{figure}
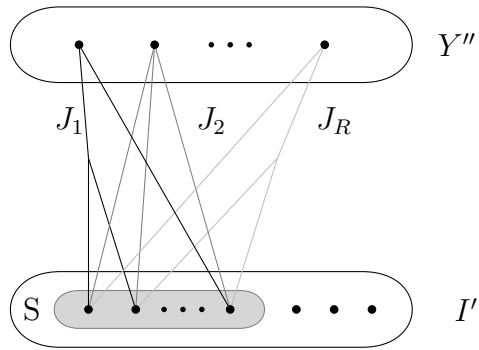

Recall from \eqref{I' properties} that 
\begin{equation*}
    |I'| \geq t\tau^R \quad \text{ and } \quad \alpha(N_{G''}^r(z) \cap Y'') \geq \frac{\alpha(Y) (R+1)}{2c^{r+1}C(R+1)} \geq \frac{\alpha(Y)}{\nu_R} \text{ for every } z\in I'.
\end{equation*} 
Therefore, we can apply Claim~\ref{No small hitting sets} to $G''$ to obtain a set $S \subseteq I'$ and disjoint induced subgraphs $J_1, \dots, J_R$ such that $|S|=t, |J_i|\leq rt$, and $S \subseteq N_{G''}(J_i)$ for all $i \in [R]$. See Figure~\ref{fig: Disjoint Induced Subgraphs} for a visualization. Let $\Lambda$ be a graph with vertex set $J_1, \dots, J_R$, where two distinct vertices $J_i$ and $J_j$ are adjacent if and only if $J_i$ and $J_j$ are not anticomplete in $G''$.

\sta{\label{No t pairwise disjoint sets} $\Lambda$ has no stable set of size $t$.}
For the sake of contradiction, suppose that $\Lambda$ has a stable set of size $t$. Then, reindexing if necessary, there are $t$ disjoint pairwise anticomplete induced subgraphs $J_1, \dots, J_t$ of $G''$. Let $\Gamma$ be the graph obtained by deleting all vertices except $\left(\bigcup_{i=1}^t J_i\right) \cup S$ and contracting all the edges in each $J_i$ to a vertex $j_i$. 
Since $J_i$ and $J_j$ are anticomplete for $i\neq j$, the set $\bigcup_{i \in [t]}\{j_i\}$ is stable, and since $S \subseteq N_{G''}(J_i)$, each $j_i$ is complete to $S$. 
Therefore, $\Gamma$ is an induced minor of $G$ isomorphic to $K_{t,t}$, a contradiction. This proves \eqref{No t pairwise disjoint sets}.
\\

By Theorem~\ref{thm:ramsey}, it follows that $\Lambda$ contains a clique of size $K$.
Reindexing if necessary, there are $K$ pairwise disjoint subsets $J_1, \dots, J_{K}$ in $G''$ such that there is at least one edge between any two sets. 
By Lemma~\ref{lem: disjointsubsets}, $G''$ contains a $K_{t,t}$ subgraph, which contradicts \eqref{claim bipartite}.   This completes the proof of \cref{thm: TI5 generalization alpha}.
\end{proof}

Next, we  provide an alternative formulation of \cref{thm: TI5 generalization alpha} in terms of 
cardinalities rather than independence numbers.

\begin{theorem} \label{thm: TI5 generalization size}
     There exists a function $M(C,s,t,r)$ that is polynomial in the parameters $s$ and $C$ such that the following holds: Let $r, s,t\in \N$, let $G\in \class$ with $\omega(G)<s$ and $C\geq 2$. For any $Y \subseteq V(G)$, there is a subset $X \subseteq V(G)$ with $|X| \leq M(C,s,t,r)$ such that if
    $$Z= \left\{z\in V(G)\sm X\  \colon |N^{r}_{G\sm X}[z]\cap Y| \geq \frac{|Y|}{C}\right\}$$
    then $\min(|Y|, |Z|) \leq M(C,s,t,r)$.  
\end{theorem}

\begin{proof}
Let $M'(C,s,t,r)$ be the function given by Theorem~\ref{thm: TI5 generalization alpha}, which is polynomial in the parameters $s$ and $C$. By Lemma~\ref{lem: poly chi bound}, there is a polynomial $p_t$ such that $\chi(G) \leq p_t(s)$.
Set $M(C,s,t,r) = p_t(s) \cdot M'(Cp_t(s),s,t,r)$, which is polynomial in $s$ and $C$. We may assume that $|Y| > M(C,s,t,r) $, as otherwise there is nothing to show.

By Theorem~\ref{thm: TI5 generalization alpha}, there is a set $X \subseteq V(G)$ with $\alpha(X) \leq M'(Cp_t(s),s,t,r)$ such that if we define
$$Z_\alpha = \left\{z\in V(G)\sm X \colon \alpha(N^{r}_{G\sm X}[z]\cap Y) \geq \frac{\alpha(Y)}{p_t(s)C}\right\}$$
then $\min(\alpha(Z_\alpha), \alpha(Y) )\leq M'(Cp_t(s),s,t,r)$. Since $\alpha(Y )>M'(Cp_t(s),s,t,r)$, it follows that  $\alpha(Z_\alpha) \leq M'(Cp_t(s),s,t,r)$. 

Let 
$$Z = \left\{z\in V(G)\sm X \colon |N^{r}_{G\sm X}[z]\cap Y| \geq \frac{|Y|}{C}\right\}.$$

We show that $X$ and $Z$ satisfy the conclusion of Theorem~\ref{thm: TI5 generalization size}. Let $z\in Z$ then, by Observation~\ref{obs: size bound by alpha and chromatic number}, it holds that 
\begin{equation*}
   \chi(N^r_{G\sm X}[z] \cap Y)\cdot \alpha(N^r_{G\sm X}[z] \cap Y) \geq |N^r_{G\sm X}[z] \cap Y| \geq \frac{|Y|}{C}\geq \frac{\alpha(Y)}{C}.
\end{equation*}
Since $N^r_{G\sm X}[z] \cap Y$ is an induced subgraph of $G$, it has chromatic number at most $p_t(s)$. Rearranging the inequality yields 
\begin{equation*}
    \alpha(N^r_{G\sm X}[z] \cap Y) \geq \frac{\alpha(Y)}{p_t(s) C},
\end{equation*}
and we conclude that $Z \subseteq Z_\alpha$. 

By Observation~\ref{obs: size bound by alpha and chromatic number}, we conclude
$$|Z| \leq  |Z_\alpha| \leq \chi(Z_\alpha) \cdot \alpha(Z_\alpha) \leq p_t(s)\cdot M'(Cp_t(s),s,t,r) \leq M(C,s,t,r) $$
and 
$$|X| \leq  \chi(X) \cdot \alpha(X) \leq p_t(s)\cdot M'(Cp_t(s),s,t,r) \leq M(C,s,t,r), $$
as required.
\end{proof}

Finally, we restate and prove \cref{theLemma}, which is another reformulation of
\cref{thm: TI5 generalization size}.
\theLemma*

\begin{proof}
Let $f(Cs)= 2M(Cs,Cs,t,r)$
where $M$ is defined as in \cref{thm: TI5 generalization size}, let $X'$ be obtained by \cref{thm: TI5 generalization size} and let $Z$ be the corresponding set, and let $K\in \set{Y,Z}$ such that $|K|\leq M(C,s,t,r)$.
Let $X=X'\cup K$; then $|X|\leq f(Cs)$. If $K=Y$ then $|N^{r}_{G\sm X}[v]\cap Y|=0$, otherwise, $K=Z$ and therefore $|N^{r}_{G\sm X}[v]\cap Y| \leq \frac{|Y|}{C}$ for every $v\in G\sm X$.
\end{proof}

\section{Coarse Balanced Separators and Tree Decompositions}

In this section  we prove \cref{thm: general coarse bal sep thm}, which we restate for convenience.
\mainthm*

The following lemma immediately implies \cref{thm: general coarse bal sep thm} by setting $\hat{Y}=\mt$. Its proof is similar to that of Lemma 9 in \cite{CH25Coarse}.

\begin{lemma}
    Let $k,s,t,r \in \N$ and let 
     $f=f_{2r,t}$ be the function given by \cref{theLemma}. Let $G\in \class$ be $(k,r)$-separable such that $\omega(G)<s$. Then for every $\hat{Y} \subseteq V(G)$, where 
    $|\hat{Y}| \leq 12f(8ks)$, there exists an internally $( 24f(8ks), r)$-coverable tree decomposition of $G$ with a bag containing $N^{r}_G[\hat{Y}]$.
\end{lemma}

\begin{proof}
    We proceed by induction on $|V(G)|$. Let $\hat{G}$ be a minimum $r$-cover of $G$.
    If $G$ is $(24f(8ks),r)$-coverable, then the tree decomposition of $G$ where $V(G)$ is in a single bag satisfies the statement.  Therefore, we may assume that  $|\hat{G}| > 24f(8ks)$.
    Note that if $|V(G)|\leq 24f(8ks)$, then we immediately have that $G$ is $(24f(8ks),r)$-coverable, which proves the base case of the induction.

    By adding vertices to $\hat{Y}$ if necessary, we may assume that $|\hat{Y}| = 12f(8ks)$. Note that $|\hat{G}| > \frac{|\hat{Y}|}{2}$. 
Let $Z_G$ and $Z_Y$ be the sets obtained by applying \cref{theLemma} to both $\hat{Y}$ and $\hat{G}$ with the parameters $C=8k$ and $r'=2r$. Let $Z = Z_G \cup Z_Y$, $G' = G \sm Z$, $\hat{G}' = \hat{G} \sm Z$, and $\hat{Y}' = \hat{Y} \sm Z$. Then, $|Z|\leq 2f(8ks)$ and for every $v \in V(G')$, we have that 
\begin{equation*}
    \left|N^{2r}_{G'}[v] \cap \hat{Y}' \right| \leq \frac{|\hat{Y}|}{8k} \quad \text{ and } \quad  \left|N^{2r}_{G'}[v] \cap \hat{G}'  \right| \leq \frac{|\hat{G}|}{8k}.
\end{equation*}

Since $G'$ is an induced subgraph of $G$, $G'$ admits $(k,r)$-balanced separators. Therefore, by taking the union of two such separators, there exists $\hat{S}$ such that $|\hat{S}|\leq 2k$ and $S=N_{G'}^r[\hat{S}]$ is a balanced separator of both $\hat{G'}$ and $\hat{Y}'$ in $G'$.
Let $\hat{Y}_S = \hat{Y}'\cap N^{2r}_{G'}[\hat{S}]$ and $\hat{G}_S = \hat{G}'\cap N^{2r}_{G'}[\hat{S}]$. Then 
\begin{equation*}
    |\hat{Y}_S| \leq 2k\cdot \frac{|\hat{Y}|}{8k} = \frac{|\hat{Y}|}{4} 
    \quad \quad \text{ and } \quad \quad |\hat{G}_S| \leq 2k \cdot \frac{|\hat{G}|}{8k} = \frac{|\hat{G}|}{4}.
\end{equation*}

Let $C$ be a connected component of $G' - S =G-(Z\cup S)$. Define 
\begin{align*}
    &C^\star = G \big[C \cup S \cup Z \cup N^{r}_{G'}[\hat{Y}_S] \big] \quad \quad  \text{and}  \quad \quad    \hat{Y}^\star = (C \cap \hat{Y} ) \cup \hat{S} \cup Z \cup \hat{Y}_S.
\end{align*}

Observe that $\hat{Y}^\star \subseteq V(C^\star)$. Next, we verify that $C^\star$ and $\hat{Y}^\star$ are suitable for the induction hypothesis.

\sta{\label{claim: C^star is not all of G} It holds that $|V(C^\star)| < |V(G)|$. }

Since $\hat{G}$ was chosen to be a minimum $r$-cover of $G$, in order to prove \eqref{claim: C^star is not all of G}, it suffices to prove that $C^\star$ can be covered by fewer than $|\hat{G}|$ balls of radius $r$. We check that balls centered on 
\begin{equation*}
    \hat{C}^\star = (C \cap \hat{G} ) \cup \hat{S} \cup Z \cup \hat{G}_S \cup \hat{Y}_S
\end{equation*} have the required property.
 Since $S$ is a balanced separator for $\hat{G}$, we have \begin{align*}
    |\hat{C}^\star| \leq |V(C) \cap \hat{G}| + |\hat{S}| + |Z| + |\hat{G}_S| + |\hat{Y}_S|
    \leq \frac{1}{2}|\hat{G} | +  2k + 2f(8ks) + \frac{|\hat{G}|}{4}  +\frac{|\hat{Y}|}{4}  
    <  |\hat{G}|.
\end{align*}
In the last inequality, we used that $f(x)\geq x$.

Now it remains to check that  $V(C^\star) \subseteq N^{r}_G[\hat{C}^\star]$. See Figure~\ref{fig: Balanced Separator Visualization} for a visualization. 
Suppose for a contradiction that $v \in C^\star$ but $v\notin N^{r}_G[\hat{C}^\star]$. If $v\in S \cup Z\cup N_{G'}^r[\hat{Y}_S]$, then evidently $v\in  N^r_G[C^\star]$, and so we may assume $v \in C$.
Since $G \subseteq N^r_G[\hat{G}]$, there exists  $u\in \hat{G}$ such that $\dist_G(u,v) \leq r$. If $u \in C$, then we are done, so we may assume $u \notin C$. 
 Since $C$ is a component of $G - (Z \cup S)$, there is a $(v,u)$-path $P_{v,u}$ of length at most $r$ intersecting $S \cup Z$. If there exists a vertex $x \in P_{v,u} \cap Z$, then $\dist(v, x) \leq r$ and so $v \in N^r_G[Z]$, which is a contradiction. Therefore, $P_{v,u}\cap S$ is non-empty. This implies that $\dist_{G'}(u,\hat{S}) \leq \dist_{G'}(u,S) + r \leq 2r$ and we deduce that $u \in \hat{G}_S$, and so $v \in N^r_G[\hat{G}_S]$.
This contradiction proves \eqref{claim: C^star is not all of G}.
\\

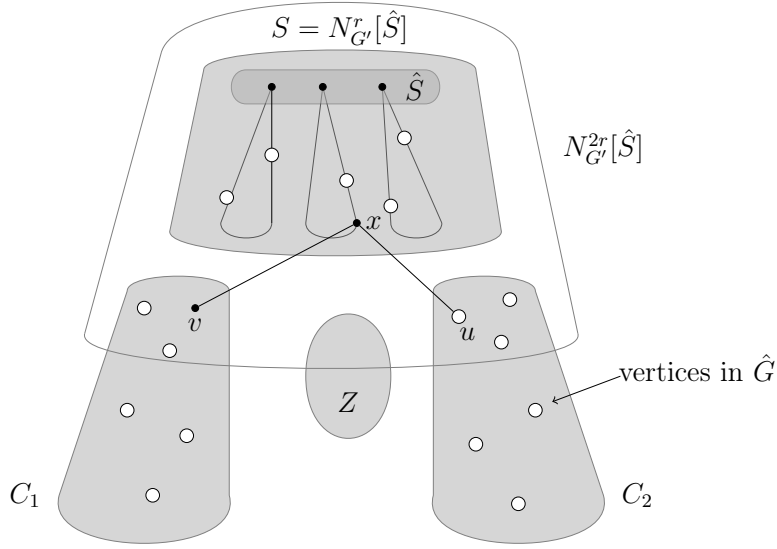
\begin{figure}[ht]
        \centering
        \scalebox{0.9}{\begin{tikzpicture}
	\begin{pgfonlayer}{nodelayer}
		\node [style=vertex] (0) at (-1.75, 5) {};
		\node [style=vertex] (1) at (-0.25, 5) {};
		\node [style=vertex] (2) at (1.5, 5) {};
		\node [style=none] (3) at (-3.25, 1) {};
		\node [style=none] (4) at (-1.75, 1) {};
		\node [style=none] (5) at (-0.75, 1) {};
		\node [style=none] (6) at (0.75, 1) {};
		\node [style=none] (7) at (1.75, 1) {};
		\node [style=none] (8) at (3.25, 1) {};
		\node [style=none] (10) at (-4.75, 0.75) {};
		\node [style=none] (11) at (5, 0.75) {};
		\node [style=none] (12) at (2.75, 5.5) {};
		\node [style=none] (13) at (-2.5, 4.5) {};
		\node [style=none] (14) at (2.75, 4.5) {};
		\node [style=none] (15) at (-3.75, 5.5) {};
		\node [style=none] (16) at (4.25, 5.5) {};
		\node [style=none] (17) at (-2.5, 5.5) {};
		\node [style=none] (19) at (2.45, 5.025) {$\hat{S}$};
		\node [style=none] (21) at (0.25, 6.75) {$S = N^r_{G'}[\hat{S}]$};
		\node [style=none] (22) at (3, -1) {};
		\node [style=none] (23) at (6, -1) {};
		\node [style=none] (24) at (3, -7) {};
		\node [style=none] (25) at (8, -7) {};
		\node [style=none] (26) at (-3, -1) {};
		\node [style=none] (27) at (-6, -1) {};
		\node [style=none] (28) at (-3, -7) {};
		\node [style=none] (29) at (-8, -7) {};
		\node [style=none] (31) at (9, -7) {$C_2$};
		\node [style=none] (34) at (-9, -7) {$C_1$};
		\node [style=vertex set] (35) at (1.75, 1.5) {};
		\node [style=vertex set] (36) at (-1.75, 3) {};
		\node [style=vertex set] (37) at (-3.075, 1.75) {};
		\node [style=vertex set] (39) at (-5.5, -1.5) {};
		\node [style=vertex set] (40) at (-4.75, -2.75) {};
		\node [style=vertex set] (41) at (3.75, -1.75) {};
		\node [style=vertex set] (43) at (5, -2.5) {};
		\node [style=vertex set] (44) at (5.25, -1.25) {};
		\node [style=none] (45) at (-5, 6) {};
		\node [style=none] (47) at (7.25, -2.25) {};
		\node [style=none] (48) at (5.5, 6) {};
		\node [style=none] (49) at (-7.25, -2.25) {};
		\node [style=none] (50) at (8, 3.25) {$N^{2r}_{G'}[\hat{S}]$};
		\node [style=vertex set] (51) at (0.45, 2.25) {};
		\node [style=vertex set] (52) at (2.15, 3.5) {};
		\node [style=vertex set] (53) at (-6, -4.5) {};
		\node [style=vertex set] (54) at (-5.25, -7) {};
		\node [style=vertex set] (55) at (-4.25, -5.25) {};
		\node [style=vertex set] (56) at (4.25, -5.5) {};
		\node [style=vertex set] (57) at (5.5, -7.25) {};
		\node [style=vertex set] (58) at (6, -4.5) {};
		\node [style=none] (59) at (-0.75, -3.5) {};
		\node [style=none] (60) at (1.75, -3.5) {};
		\node [style=none] (61) at (1.75, -3.5) {};
		\node [style=none] (62) at (0.5, -3.5) {};
		\node [style=none] (63) at (0.5, -4.25) {$Z$};
		\node [style=vertex] (65) at (-4, -1.5) {};
		\node [style=none] (67) at (-4, -2) {$v$};
		\node [style=vertex] (68) at (0.75, 1) {};
		\node [style=none] (69) at (1.25, 1) {$x$};
		\node [style=none] (70) at (4, -2.25) {$u$};
		\node [style=none] (71) at (6.5, -4.25) {};
		\node [style=none] (72) at (8.5, -3.5) {};
		\node [style=none] (73) at (10.75, -3.25) {vertices in $\hat{G}$};
	\end{pgfonlayer}
	\begin{pgfonlayer}{edgelayer}
		\draw [style=grey fill] (61.center)
			 to [bend left=90, looseness=2.50] (59.center)
			 to [bend left=90, looseness=2.50] cycle;
		\draw [style=simple] (0) to (3.center);
		\draw [style=simple] (1) to (5.center);
		\draw [style=simple] (1) to (6.center);
		\draw [style=simple] (2) to (7.center);
		\draw [style=simple] (8.center) to (2);
		\draw [style=simple, bend right=90] (5.center) to (6.center);
		\draw [style=simple, bend right=90] (3.center) to (4.center);
		\draw [style=simple, bend right=90] (7.center) to (8.center);
		\draw [style=grey fill] (17.center)
			 to [bend right=90, looseness=1.50] (13.center)
			 to (14.center)
			 to [bend right=90, looseness=1.50] (12.center)
			 to cycle;
		\draw [style=grey fill] (16.center)
			 to [bend right=90, looseness=0.25] (15.center)
			 to (10.center)
			 to [bend right=90, looseness=0.25] (11.center)
			 to cycle;
		\draw [style=grey fill] (22.center)
			 to (24.center)
			 to [bend right=105] (25.center)
			 to (23.center)
			 to [bend left=270, looseness=0.50] cycle;
		\draw [style=grey fill] (27.center)
			 to [bend right=270, looseness=0.50] (26.center)
			 to (28.center)
			 to [bend left=105] (29.center)
			 to cycle;
		\draw [style=gray edge] (49.center) to (45.center);
		\draw [style=gray edge, in=105, out=75, looseness=0.50] (45.center) to (48.center);
		\draw [style=gray edge] (48.center) to (47.center);
		\draw [style=gray edge, bend left=105, looseness=0.25] (47.center) to (49.center);
		\draw [style=simple] (68) to (65);
		\draw [style=simple] (68) to (41);
		\draw [style=diredge] (72.center) to (71.center);
		\draw (0) to (4.center);
	\end{pgfonlayer}
\end{tikzpicture}}
        \caption{Visualization of the proof of \eqref{claim: C^star is not all of G}. The vertices in $\hat{G}$ are distinguished by white fill and black outline.}
        \label{fig: Balanced Separator Visualization}
        \end{figure}
%
%
%
%
Since $S$ is a balanced separator for $\hat{Y}$, we can bound $|\hat{Y}^\star|$ as follows: 
\begin{align*}
    |\hat{Y} ^\star| &\leq |V(C) \cap \hat{Y}| + |\hat{S}| + |Z| + |\hat{Y}_S|\\
    &\leq \frac{1}{2}|\hat{Y} | +  2k + 2f(8ks) + \frac{|\hat{Y}|}{4}\\ 
    &< \frac{3|\hat{Y}|}{4} + 3f(8ks) \leq  12f(8ks) =|\hat{Y} |.
\end{align*}
%
%
Now suppose that $C_1, \dots, C_m$ are the components of $G -(Z\cup S)$. 
 By the inductive hypothesis, for every $i \in [m]$, the graph $C_i^\star$ admits a tree-decomposition $\cT_i=(T_i, \beta_i)$ such that every bag is internally $(24f(8ks),r)$-coverable and there exists a bag $\beta_i(t_i)$ that contains $N^{r}_{C^\star_i}[\hat{Y}^\star_i]$.

Let $T$ be obtained from the disjoint union of $T_1, \dots, T_m$ by adding a new vertex $t_0$ that is adjacent to $t_1, \dots, t_m$. Define $\beta(t) = \beta_i(t)$ for all $t\in T_i$ and $i\in [m]$, and set $\beta(t_0) = N^{r}_G[\hat{Y}] \cup S\cup Z$. Then $\cT=(T,\beta)$ is a tree decomposition of $G$. 
To see this, note that for any $i<j \in [m]$ it holds that $C_i^\star \cap C_j^\star \subseteq N_G^r[\hat{Y}_S] \cup S\cup Z$, which is contained in the bag $\beta(t_0)$.



We observe that the set $N^r_G[\hat{Y}] \cup S \cup Z$ can be $r$-covered by the set $\hat{Y} \cup \hat{S} \cup Z$, which has at most $12f(8ks) + 2k + 4f(8ks) \leq 24f(8ks)$ vertices. Since the other bags do not change, every bag of $\cT$ is internally $(24f(8ks), r)$-coverable and $\cT$ has a bag containing $N_G^r[\hat{Y}]$, as required.
\end{proof}

\section{Acknowledgments}
We are grateful to Daniel Lokshtanov and Ajaykrishnan E. S. for many inspiring discussions and to Alex Divoux for proofreading parts of this paper.

\bibliographystyle{plain}
\bibliography{references}
\end{document}